\documentclass[oneside,11pt]{amsart}

\makeatletter
\@namedef{subjclassname@2020}{
  \textup{2020} Mathematics Subject Classification}
\makeatother
\usepackage[a4paper,width=170mm,top=27mm,bottom=27mm]{geometry}
\usepackage{cases}
\usepackage{amsmath,amssymb}
\usepackage[hidelinks]{hyperref}
\usepackage{esint}

\usepackage[french,english]{babel}

\newtheorem{theorem}{Theorem}[section]
\newtheorem{lemma}[theorem]{Lemma}

\newtheorem{corollary}[theorem]{Corollary}

\theoremstyle{definition}

\newcommand{\R}{\mathbb{R}}
\newcommand{\C}{\mathbb{C}}

\newcommand{\vare}{\varepsilon}

\newcommand{\pt}{\partial}

\newcommand{\Z}{{\mathbb{Z}}}

\numberwithin{equation}{section}

\begin{document}
\address{Yongming Luo
\newline \indent
Faculty of Computational Mathematics and Cybernetics
\newline \indent Shenzhen MSU-BIT University, China}
\email{luo.yongming@smbu.edu.cn}

\address{Jiayu Zheng
\newline \indent
Faculty of Computational Mathematics and Cybernetics
\newline \indent Shenzhen MSU-BIT University, China}
\email{jyzheng@smbu.edu.cn}

\title[Unconditional uniqueness for the cubic NLS]
{Unconditional uniqueness for the cubic nonlinear Schr\"odinger equation in $\dot{H}^{\frac12}(\R^3)$}
\author{Yongming Luo and Jiayu Zheng}

\begin{abstract}
While unconditional uniqueness for the cubic nonlinear Schr\"odinger equation (NLS) on $\R^3$ is known in the scaling-subcritical Sobolev
spaces $\dot H^s(\R^3)$ for $1/2<s<1$, the critical case $s=1/2$ has remained open. We resolve this problem by adapting Kato's
bootstrap argument to the critical space $\dot B^{-1}_{\infty,\infty}$. The main difficulty at this endpoint is that the natural solution class does not provide the coefficient regularity required to close the bootstrap in $\dot B^{-1}_{\infty,\infty}$. To overcome this issue, we perform a second Duhamel iteration of Born type and estimate the resulting double Duhamel integral as a whole. This yields improved estimates that allow the bootstrap to close in $\dot B^{-1}_{\infty,\infty}$ using only the available regularity of the coefficients.
\end{abstract}


\maketitle

\section{Introduction}
In this paper, we study the unconditional uniqueness of solutions to the cubic nonlinear Schr\"odinger equation (NLS)
\begin{equation}\label{cnls}
(i\pt_t +\Delta) u=\mu |u|^2 u,\qquad \mu=\pm 1
\end{equation}
on $\R^3$. The equation is invariant under the scaling $u(t,x)\mapsto \lambda u(\lambda^2t,\lambda x)$, which preserves the $\dot H^{1/2}(\R^3)$-norm of the initial data. In \cite{Cazenave1990}, Cazenave and Weissler considered the Cauchy problem for \eqref{cnls} and constructed solutions in a suitable function space, which combines time continuity in $\dot H^{1/2}(\R^3)$ with additional space-time integrability that provides sufficient control on the nonlinear terms for closing the contraction proof. However, the local theory from \cite{Cazenave1990} does not directly rule out another distributional solution with the same initial data which is only known to belong to $C(I;\dot H^{1/2}(\R^3))$. This leaves the question of whether the additional space-time assumptions can be removed, so that uniqueness holds in the natural Sobolev class itself. Such a property is usually referred to as \textit{unconditional uniqueness}.

An important motivation for studying unconditional uniqueness is to identify solutions obtained by different constructions. Indeed, the auxiliary
spaces used in two local well-posedness arguments need not coincide, and uniqueness within each of these spaces does not immediately imply
that the resulting solutions agree. One would first have to verify that both solutions belong to a common class in which uniqueness is
known. Unconditional uniqueness removes this additional requirement by establishing uniqueness in the natural Sobolev class, independently
of the auxiliary conditions imposed in the construction.

The unconditional uniqueness property of a solution is also relevant to applications. As pointed out by Chen and Holmer \cite{uncond2}, a motivation in the periodic setting comes from control problems associated with the Lugiato-Lefever system introduced in \cite{LugiatoLefever1987}. Therein, the authors consider the
forced NLS equation
\[
i\partial_t u_f+\Delta u_f
=
\pm |u_f|^{p-1}u_f+f,
\qquad
u_f(0)=u_0,
\qquad x\in\mathbb T^n,
\]
where the forcing $f$ and the initial data $u_0$ are chosen so that the corresponding solution minimizes a prescribed functional.
The choice of the solution class plays an important role in such a problem. While variational methods are naturally formulated in
an $L^2$ or $H^1$ framework, the available well-posedness theory may require higher regularity or additional space-time bounds. Searching
for minimizers subject to these auxiliary conditions can therefore be more restrictive than working in the natural Sobolev class.
An unconditional uniqueness result for the relevant equation allows a solution obtained in the latter class to be identified with the
one supplied by the local theory, without separately verifying the auxiliary conditions. In this way, unconditional uniqueness helps
combine the solution class used in the evolution theory with the one in which the variational problem is formulated.

A systematic study of unconditional uniqueness for NLS on $\R^d$ was initiated by Kato \cite{Kato1995}. Therein, the uniqueness argument is based on a bootstrap for the difference of two solutions with the same initial data. The zero initial value eliminates the free term in Duhamel's formula, while Strichartz estimates and estimates for the nonlinear difference give a bound by a small multiple of the same norm on a sufficiently short interval. An essential point is that the norm to be absorbed must already be finite under the assumed Sobolev regularity, rather than imposed as an additional hypothesis on the solutions.

At low regularity\footnote{In other words, for NLS with small critical scaling $s$.}, Kato's approach encounters a major difficulty in the spatial integrability required by the nonlinear estimates. For the present equation, the Sobolev embedding $\dot H^{1/2}(\R^3)\hookrightarrow L^3(\R^3)$ only gives $|u|^2u-|v|^2v\in L_t^\infty L_x^1$ from the natural bounds on $u$ and $v$, which is not a function space covered by the usual admissible Strichartz estimates. More precisely, the standard dual spatial exponents in three dimensions range from $6/5$ to $2$, whereas an $L^1$-bound does not imply an $L^{6/5}$-bound. Thus a direct combination of Sobolev embedding, H\"older inequality and Strichartz estimates does not suffice at the critical regularity.

A major breakthrough was made by Furioli and Terraneo \cite{FurioliTerraneo2003}, where the difference is estimated in a homogeneous Besov space of negative order, in which case the paraproduct theory makes it possible to trade differentiability for a larger spatial exponent. For example, along the Sobolev embedding relation $s-3/2=\sigma-3/\rho$, lowering $\sigma$ makes it possible to enlarge the value of $\rho$ in $\dot B^\sigma_{\rho,2}$, which in turn gives additional room for matching the nonlinear estimates with Strichartz. As a consequence, the results from \cite{FurioliTerraneo2003} also imply that the solution of the cubic NLS is unconditionally unique in $\dot H^{s}(\R^3)$ for any $s\in(1/2,1)$.

Nevertheless, such gain in integrability comes at the expense of additional regularity on the coefficient. To be more precise, multiplication by a coefficient on a negative-order space requires positive fractional regularity of that coefficient. The choice of the negative index must therefore balance the gain in the spatial exponent against the regularity available for the nonlinear coefficients. 

In a later work \cite{HanFang2013}, Han and Fang further developed the approach from \cite{FurioliTerraneo2003} using the extended inhomogeneous Strichartz estimates of Foschi \cite{Foschi2005} as well as Vilela \cite{Vilela2007}, which gives more freedom in the choice of the negative Sobolev or Besov space and enlarges the range of dimensions, regularities and nonlinear powers. They also exploit an important observation: the difference of two solutions can have better space-time regularity than is known for either solution separately. Indeed, the difference has zero initial data, hence its Duhamel representation can provide auxiliary bounds which do not follow from Sobolev embedding alone.

For further unconditional uniqueness results of the NLS on $\R^d$, we refer to \cite{Rogers2007,LuXu2016}. We also refer to \cite{GuoKwonOh2013,HerrSohinger2019,uncond1,uncond2,uncond3,Kishimoto2021,OhWang2021} for unconditional uniqueness results of the NLS equation in the periodic setting.

The above results do not cover the three-dimensional cubic equation at $s=1/2$. To better understand the difficulties in the critical setting, we may examine in more details the function $w=u-v$, namely the difference of two solutions with the same initial data. The difference equation can then be written as
$(i\partial_t+\Delta)w=aw+b\overline w$, where $a$ and $b$ are quadratic expressions in $u,v$ and their conjugates. The assumption
$u,v\in C_t\dot H_x^{1/2}$ gives
\[
a,b\in L_t^\infty L_x^{3/2},
\qquad
|D|^{1/2}a,\ |D|^{1/2}b
\in L_t^\infty L_x^{6/5}.
\]
These bounds, however, are insufficient for the negative-order multiplication argument described above. Indeed, consider estimating $aw$ in $\dot W^{-1/2,6/5}$ with $w$ measured in $\dot W^{-1/2,6}$. By duality, one needs to estimate $|D|^{1/2}(a\phi)$ in $L^{6/5}$ for a test function $\phi$ satisfying $|D|^{1/2}\phi\in L^6$, and the fractional Leibniz rule shall produce the term $\||D|^{1/2}a\|_{L^{6/5}}\|\phi\|_{L^\infty}$. Controlling this term by the available norms would require the embedding $\|\phi\|_{L^\infty}\lesssim\||D|^{1/2}\phi\|_{L^6}$ which is in general not true. Thus the gain obtained by estimating the difference in a negative-order space is accompanied by a multiplication estimate, where the available coefficient bounds no longer justify. Unlike in the previous work \cite{Kato1995,FurioliTerraneo2003,HanFang2013}, shortening the time interval also does not compensate the spatial estimate in the critical setting, and a different treatment of the coefficient terms is needed to close the uniqueness argument.

Our main unconditional uniqueness result is as follows.

\begin{theorem}\label{main thm}
Let $J\subset \R$ be an interval and suppose that $u,v\in C(J; \dot{H}^{\frac12}(\R^3))$ are distributional solutions of the cubic NLS \eqref{cnls} satisfying $u(t_0)=v(t_0)$ for some $t_0\in J$. Then $u(t)=v(t)$ for any $t\in J$.
\end{theorem}

We briefly explain the strategy for proofing Theorem \ref{main thm}. Inspired by \cite{FurioliTerraneo2003,HanFang2013}, we also aim to find a negative-order space in which we are able to close the estimates for the difference equation. In the present proof, the appropriate choice is $\dot B^{-1}_{\infty,\infty}(\R^3)$. Besides being invariant under the scaling of \eqref{cnls}, this space is naturally suggested by the forcing term. To see this, write first $S(t)=e^{it\Delta}$ and $D_\tau G(t)=\int_\tau^t S(t-s)G(s)\,ds$. On a dyadic spatial frequency $N$, the elementary frequency-localized bound and the dispersive estimate give
\[
\|P_NS(t)\|_{L_x^1\to L_x^\infty}
\lesssim \min\{N^3,|t|^{-3/2}\},
\qquad
\int_0^\infty\min\{N^3,t^{-3/2}\}\,dt\lesssim N
\]
(see Lemma \ref{lem 3.1} below for details). Consequently, a forcing in $L_t^\infty L_x^1$ gives a uniform bound for $N^{-1}\|P_ND_\tau G(t)\|_{L_x^\infty}$. Taking the supremum over the dyadic frequencies yields precisely the $\dot B^{-1}_{\infty,\infty}$-norm. It is also worth pointing out that this choice matches the refined Sobolev inequality of G\'erard--Meyer--Oru \cite{GMO96},
\[
\|w\|_{L_t^\infty L_x^3}^3
\lesssim
\|w\|_{L_t^\infty\dot H_x^{1/2}}^2
\|w\|_{L_t^\infty\dot B^{-1}_{\infty,\infty}},
\]
stated as Lemma \ref{lem:refined-sobolev} below. The cubic forcing can therefore be estimated by a small multiple of the same Besov norm whenever the difference starts from zero. These two estimates, rather than scaling alone, determine the role of $\dot B^{-1}_{\infty,\infty}$ in our argument.

The mixed terms present a further difficulty. At a fixed time $\tau$, let $U$ be the strong solution obtained from the local theory \cite{Cazenave1990} with initial value $U(\tau)=u(\tau)$, and set $w=u-U$. On a sufficiently short interval $J$ starting at $\tau$, the local theory gives $U\in C_t\dot H_x^{1/2}\cap L_{t,x}^5$ and $|D|^{1/2}U\in L_t^{5/2}L_x^{30/7}$. Expanding the nonlinearity, we write
\[
(i\partial_t+\Delta)w=Aw+B\overline w+G,
\qquad G=\mu|w|^2w,
\qquad w(\tau)=0,
\]
where $A$ and $B$ are quadratic expressions in $U,w$ and their conjugates, and every term contains at least one factor of $U$. These coefficients belong to $L_t^5L_x^{15/8}$ and have $3/10$ of a spatial derivative in $L_t^{5/2}L_x^2$. Nevertheless, such bounds do not directly lead to small multiplication operators on $L_t^\infty\dot B^{-1}_{\infty,\infty}$. In particular, estimating $Aw$ by H\"older will actually use the $L_t^\infty L_x^3$-norm of $w$, but not its negative Besov norm. We thus encounter again the regularity cost of the negative-order approach.

Such difficulty will be overcome by employing a suitable Born iteration argument. Indeed, the mixed terms $Aw+B\overline w$ can be viewed as a time-dependent perturbation of the free Schr\"odinger evolution. Such point of view will in turn connects our
problem with the dispersive theory of Schr\"odinger equations with potentials, where expansions of the perturbed evolution allow one
to exploit the free propagators between successive potential interactions. Such compositions play an important role in the work
of Journ\'e, Soffer and Sogge \cite{JSS1991}, and are also studied by Rodnianski and Schlag \cite{RodnianskiSchlag2004} through Born
expansions of the resolvent and, for time-dependent potentials, iterations of Duhamel's formula.

In the following, we explain the Born expansion argument applied in the current setting in more details. Define first the operator $K$ by 
$$Kf:=-iD_\tau(Af+B\overline f).$$ 
The difference equation gives $w=-iD_\tau G+Kw$. With $r=Kw$ and $z_0=K(-iD_\tau G)$, another application of $K$ yields
\[
r=z_0+Kr.
\]
We shall use two different exotic Strichartz estimates (see \eqref{2.4}) to estimate $r$ and $Kr$. The first one shows that $r\in L_{t,x}^5$ using only $w\in L_t^\infty L_x^3$, while the second makes $K$ a small operator on $L_{t,x}^5$ when the interval is sufficiently short, so that $Kr$ can be absorbed. 

It remains to control $z_0$, whose terms are double Duhamel integrals of the form $D_\tau(FD_\tau G)$ or $D_\tau(F\overline{D_\tau G})$, with $F$ equals $A$ or $B$. The main point is to estimate these integrals without separating the two free propagators from the inserted coefficient. More precisely, by direct computation, one verifies that the kernel of $S(\alpha)F S(\beta)$ can be reduced to the one of the free evolution $S(c)F$ at the effective time $c=\alpha\beta/(\alpha+\beta)$, see Lemma~\ref{lem:kernel} below for details. A change of variables to $c$, followed by H\"older and Hardy-Littlewood-Sobolev inequality, gives the key estimate of Lemma~\ref{lem:double}:
\[
\begin{aligned}
&\|D_J(FD_\tau G)\|_{L_{t,x}^5(\R)}
 +\|D_J(FD_\tau G)\|_{L_t^\infty\dot B^{-1}_{\infty,\infty}(\R)}\\
&\qquad\lesssim
 \||D|^{3/10}F\|_{L_t^{5/2}L_x^2(J)}
 \|G\|_{L_t^\infty L_x^1(J)}.
\end{aligned}
\]
Here $D_J$ denotes the global retarded extension defined below; the same estimate holds with the inner forcing conjugated. The gain is thus a bound for the iterated forcing using the available regularity of the coefficient, rather than an increase in the regularity of the coefficient itself. This is the feature of the Born iteration that allows us to avoid the endpoint multiplication estimate.

Having these estimates in hand, we first absorb $Kr$ in $L_{t,x}^5$ and then use the Besov output estimates to control $r$ in $L_t^\infty\dot B^{-1}_{\infty,\infty}$. Together with the bound for $D_\tau G$, this gives, on a sufficiently short interval,
\[
\begin{aligned}
\|w\|_{L_t^\infty\dot B^{-1}_{\infty,\infty}}
&\lesssim \|w\|_{L_t^\infty L_x^3}^3\\
&\lesssim \|w\|_{L_t^\infty\dot H_x^{1/2}}^2
          \|w\|_{L_t^\infty\dot B^{-1}_{\infty,\infty}}.
\end{aligned}
\]
Since $w(\tau)=0$ and $w\in C_t\dot H_x^{1/2}$, a further restriction of the time interval closes the bootstrap. In this way, the negative Besov space is used for the final uniqueness estimate, while the contraction for the coefficient term is carried out only after one Duhamel iteration.

\subsubsection*{Outline of the paper}
The rest of the paper is organized as follows. In Section~\ref{sec:preliminaries} we collect the refined Sobolev inequality and the standard and exotic Strichartz estimates. Section~\ref{sec:besov} is devoted to the Besov estimates for the Duhamel integral and the free Schr\"odinger evolution. In Section~\ref{sec:double-duhamel} we prove the kernel identity and the double Duhamel estimate. Finally, we prove Theorem~\ref{main thm} in Section~\ref{sec:main-proof} by the Born iteration and the short-time bootstrap described above.

\subsection{Notation}
We use the notation $A\lesssim B$ whenever there exists some positive constant $C$ such that $A\leq CB$. Similarly we define $A\gtrsim B$ and use $A\sim B$ when $A\lesssim B\lesssim A$. We use $P_N$ to denote the standard Littlewood-Paley projection for a dyadic number $N$. We use $L^p$ and $H^s$ to denote the usual Lebesgue and Sobolev spaces. We denote by $\dot{B}^{-1}_{\infty,\infty}(\R^3)$ the spatial Besov space defined via the norm $\|f\|_{\dot{B}^{-1}_{\infty,\infty}}:=\sup_{N}N^{-1}\|P_Nf\|_{L_x^\infty}$. We use $S(t)$ to denote the Schr\"odinger propagator $e^{it\Delta}$. For $0<\tau<T<\infty$ and $J=[\tau,T]$, the Duhamel integrals are defined as
\[
\begin{aligned}
D_{\tau}f(t)&=\int_\tau^t S(t-s)f(s)\,ds,\qquad t\in J,\\
D_Jf(t)&=\int_{-\infty}^t S(t-s)\chi_J(s)f(s)\,ds,\qquad t\in \R,
\end{aligned}
\]
where $\chi_J$ denotes the characteristic function associated with $J$. It is worth noting that $D_\tau f(t)$ coincides with $D_J f(t)$ for $t\in J$. Finally, a pair $(q,r)$ is called admissible if
$2\le q,r\le\infty$ and $\frac{2}{q}+\frac{3}{r}=\frac32$.

\section{Sobolev and Strichartz estimates}\label{sec:preliminaries}
As preliminaries, we collect in this section some useful auxiliary estimates which will be frequently applied in the paper.

\subsubsection*{A refined Sobolev inequality}
We record here a refined Sobolev inequality from \cite[Thm. 2]{GMO96} that bounds the Lebesgue norm of a function by the weighted product of its Sobolev and Besov norms.

\begin{lemma}[{\cite{GMO96}}]
\label{lem:refined-sobolev}
We have the following refined Sobolev inequality
\begin{equation*}
    \|f\|_{L_x^3}
    \lesssim
    \|f\|_{\dot H_x^{1/2}}^{2/3}
    \|f\|_{{\dot{B}^{-1}_{\infty,\infty}}}^{1/3}.
\end{equation*}
\end{lemma}

\subsubsection*{Standard and exotic Strichartz estimates}
\begin{lemma}[Strichartz estimates, \cite{EndpointStrichartz}]
Let $I\subset\R$ be an interval and $t_0\in\overline I$. Then for every admissible pair $(q,r)$ we have
\begin{equation*}
    \|S(t-t_0)f\|_{L_t^q L_x^r(I)}
    \lesssim
    \|f\|_{L_x^2}.
\end{equation*}
Moreover, if $(q,r)$ and $(\widetilde q,\widetilde r)$ are
admissible pairs, then
\begin{equation}
    \left\|
        \int_{t_0}^t S(t-s)F(s)\,ds
    \right\|_{L_t^qL_x^r(I)}
    \lesssim
    \|F\|_{L_t^{\widetilde q'}L_x^{\widetilde r'}(I)}.
    \label{eq:inhom-strichartz}
\end{equation}
\end{lemma}

We will also need the following exotic version of the Strichartz estimates from \cite{Foschi2005,Vilela2007}, see also \cite[Cor. 7.1]{Taggart2010}.
\begin{lemma}[Exotic Strichartz estimate, \cite{Foschi2005,Vilela2007,Taggart2010}]
\label{lem:foschi}
Let $I\subset\mathbb R$ be an interval and $t_0\in\overline I$. Then \eqref{eq:inhom-strichartz} holds provided that
\[
\begin{gathered}
1\le q,\widetilde q<\infty,
\qquad
2\le r,\widetilde r\le\infty,\\[0.3em]
\frac1q
<
3\left(\frac12-\frac1r\right),
\qquad
\frac1{\widetilde q}
<
3\left(\frac12-\frac1{\widetilde r}\right),\\[0.3em]
\frac1q+\frac1{\widetilde q}
=
\frac32
\left(
1-\frac1r-\frac1{\widetilde r}
\right),
\end{gathered}
\]
and, in addition, either
\[
\begin{gathered}
\frac1q+\frac1{\widetilde q}<1,\qquad
\frac1r\le\frac3{\widetilde r},
\qquad
\frac1{\widetilde r}\le\frac3r,
\end{gathered}
\]
or
\[
\begin{gathered}
\frac1q+\frac1{\widetilde q}=1,\qquad
\frac1r<\frac3{\widetilde r},
\qquad
\frac1{\widetilde r}<\frac3r,\qquad
\frac1r\le\frac1q,
\qquad
\frac1{\widetilde r}\le\frac1{\widetilde q}.
\end{gathered}
\]
\end{lemma}

The particular version of the exotic Strichartz estimates applied in this paper is given in the following corollary by setting $(q,r,\tilde{q},\tilde{r})\in\{(5,5,5/4,15/2),(5,5,5/3,15/4)\}$ in Lemma \ref{lem:foschi}.

\begin{corollary}
\label{cor foschi}
The following exotic Strichartz estimate hold:
\begin{align}\label{2.4}
\left\|
    \int_{t_0}^t S(t-s)F(s)\,ds
\right\|_{L_{t,x}^5(I)}
\lesssim
\min\{\|F\|_{L_t^5L_x^{15/13}(I)},\|F\|_{L_t^{5/2}L_x^{15/11}(I)}\}.
\end{align}
\end{corollary}

\section{The Besov-Strichartz type estimates}\label{sec:besov}
As one of the main issues, we note that the crucial $\dot{B}^{-1}_{\infty,\infty}$-norm appearing in the bootstrap argument is in general incompatible with the Strichartz estimates given in previous section. We solve this issue by developing suitable estimates for the Duhamel integral, in term of the $\dot{B}^{-1}_{\infty,\infty}$-norm.    

We first give a uniform in time control on the $\dot{B}^{-1}_{\infty,\infty}$-norm of the Duhamel integral associated to the term $O(w^3)$

\begin{lemma}\label{lem 3.1}
For any interval $J\subset \R$ and $g:\R\times \R^3\to \C$ it holds 
\begin{align*}
\|D_Jg\|_{L_t^\infty \dot{B}^{-1}_{\infty,\infty}(\R)}\lesssim \|g\|_{L_t^\infty L_x^1(J)}.
\end{align*}
\end{lemma}

\begin{proof}
Let $N$ be a dyadic number and $\phi$ be the bump function associated with the Littlewood-Paley projection. Then the symbol of $P_N$ is given by $\phi_N(\xi)=\phi(N^{-1}\xi)$. We can now use Young's convolution inequality and Hausdorff-Young to obtain that for $f:\R^3\to\C$ it holds
\[
\|P_N S(t)f\|_{L_x^\infty}\lesssim \|\mathcal{F}^{-1}({\phi_N e^{-it|\xi|^2}})\|_{L_x^\infty}\|f\|_{L_x^1}\lesssim \|\phi_N\|_{L_\xi^1}\|f\|_{L_x^1}
\lesssim N^3\|f\|_{L_x^1}.
\]
On the other hand, we can directly make use of the dispersive estimate to conclude that 
\[
\|P_N S(t)f\|_{L_x^\infty}\lesssim |t|^{-\frac{3}{2}}\|f\|_{L_x^1},
\]
from which we conclude that
\[\|P_N S(t)f\|_{L_x^\infty}\lesssim \min\{N^3,|t|^{-\frac{3}{2}}\}\|f\|_{L_x^1}.\]
Hence for any $t_*\in\R$ we have
\begin{equation*}
\begin{aligned}
\|P_ND_J g(t_*)\|_{L_x^\infty}&=\left\|\int_{-\infty}^{t_*}P_NS(t_*-s)\chi_Jg(s)\,ds\right\|_{L_x^\infty}\leq 
\int_{-\infty}^{t_*}\left\|P_NS(t_*-s)\chi_Jg(s)\right\|_{L_x^\infty}\,ds\\
&\lesssim \|g\|_{L_t^\infty L_x^1(J)}\int_0^\infty \min\{N^3,\tau^{-\frac{3}{2}}\}\,d\tau
\lesssim N \|g\|_{L_t^\infty L_x^1(J)}.
\end{aligned}
\end{equation*}
Dividing $N$ from both sides and taking supremum over $N\in 2^{\Z}$ and $t_*\in\R$ yields the desired claim.
\end{proof}

The following lemma shows how we are able to control the $\dot{B}^{-1}_{\infty,\infty}$-norm of a function $f$ by the space-time estimate of its Schr\"odinger linear flow.

\begin{lemma}\label{lem 3.2}
For $f:\R^3\to\C$ it holds
\begin{equation*}
\|f\|_{\dot{B}^{-1}_{\infty,\infty}}\lesssim \|S(\cdot)f\|_{L_{t,x}^5(0,\infty)}.
\end{equation*}
\end{lemma} 

\begin{proof}
Let $N$ be a given dyadic number and rewrite $P_Nf$ to $P_Nf=S(-t)P_N S(t)f$. Notice that the symbol of $S(-t)P_N$ is given by $\phi(N^{-1}\xi)e^{it|\xi|^2}$. Since $\phi$ is compactly supported in $\{\xi:|\xi|\lesssim 1\}$, the symbol is only non-vanishing for $|\xi|\lesssim N$. Notice also that due to the appearance of $e^{it|\xi|^2}$, the Mikhlin's norm of the symbol of $S(-t)P_N$ is bounded by 
\[\lesssim 1+|t||\xi|^2\lesssim 1+|t|N^2\lesssim 1\qquad \text{for $|t|\lesssim N^{-2}$.}\]
Combining the Mikhlin's multiplier theorem and Bernstein inequality, this yields
\[N^{-1}\|P_Nf\|_{L_x^{\infty}}\lesssim N^{-\frac{2}{5}}
\|S(-t)P_NS(t)f\|_{L_x^{5}}\lesssim N^{-\frac{2}{5}}\|S(t)f\|_{L_x^{5}}.\]
Now take $L_t^5$-norm over $t\in(0,N^{-2})$ and cancel the factor $N^{-\frac25}$ on both sides to obtain
\[N^{-1}\|P_Nf\|_{L_x^{\infty}}\lesssim \|S(\cdot)f\|_{L_{t,x}^{5}(0,N^{-2})}\lesssim \|S(\cdot)f\|_{L_{t,x}^{5}(0,\infty)}.\]
The desired claim follows from taking supremum over $N\in 2^{\Z}$ on the l.h.s. 
\end{proof}

Lemma \ref{lem 3.2} implies immediately the control on the $\dot{B}^{-1}_{\infty,\infty}$-norm of the Duhamel integral associated to the term $O(u^2 w)$.

\begin{corollary}\label{cor 3.3}
For any interval $J\subset \R$ and $g:\R\times \R^3\to \C$ it holds 
\begin{align*}
\|D_Jg\|_{L_t^\infty \dot{B}^{-1}_{\infty,\infty}(\R)}\lesssim \min\{\|g\|_{L_t^5L_x^{15/13}(J)},\|g\|_{L_t^{5/2}L_x^{15/11}(J)}\}.
\end{align*}
\end{corollary}

\begin{proof}
Fix some $t_*\in\R$. By Lemma \ref{lem 3.2} and Corollary \ref{cor foschi} we have
\begin{equation}\label{3.5}
\begin{aligned}
\|D_Jg(t_*)\|_{\dot{B}^{-1}_{\infty,\infty}(\R^3)}&=\|D_J(\chi_{s\leq t_*}g)(t_*)\|_{\dot{B}^{-1}_{\infty,\infty}(\R^3)}
\lesssim \|S(\cdot)D_J(\chi_{s\leq t_*}g)(t_*)\|_{L_{t,x}^5(0,\infty)}\\
&=\left\|
\int_{-\infty}^{\cdot+t_*}S(\cdot+t_*-s)(\chi_J\chi_{s\leq t_*}g)(s)\,ds
\right\|_{L_{t,x}^5(0,\infty)}\\
&\lesssim \min\{\|g\|_{L_t^5L_x^{15/13}(J)},\|g\|_{L_t^{5/2}L_x^{15/11}(J)}\},
\end{aligned}
\end{equation}
as desired. 
\end{proof}

\section{Estimation for the double Duhamel integral}\label{sec:double-duhamel}
We now turn to the estimate for the mixed terms. First, we record the kernel formula which makes the free evolution of the coefficient explicit. 
\begin{lemma}
\label{lem:kernel}
Let $\alpha,\beta\in\R\setminus\{0\}$ satisfy $\alpha+\beta\neq0$, and let $f:\R^3\to\C$ be a given function. Then the integral kernel $\mathcal K_{\alpha,\beta,f}(x,y)$ of
\(
S(\alpha)f S(\beta)
\)
is given by
\begin{equation}
\mathcal K_{\alpha,\beta,f}(x,y)
=
K_{\alpha+\beta}(x-y)
\left[
S\left(\frac{\alpha\beta}{\alpha+\beta}\right)f
\right]
\left(
\frac{\beta x+\alpha y}{\alpha+\beta}
\right),
\label{eq:kernel-identity}
\end{equation}
where
\begin{align}\label{def of dispersive kernel}
K_t(x)
=
\frac{1}{(4\pi i t)^{3/2}}
e^{i|x|^2/(4t)}
\end{align}
denotes the kernel of the linear Schr\"odinger group. Consequently, 
\begin{equation}
\|S(\alpha)f S(\beta)\|_{L^1_x\to L^5_x}
\lesssim
|\alpha+\beta|^{-9/10}|\beta|^{-3/5}
\left\|
S\left(\frac{\alpha\beta}{\alpha+\beta}\right)f
\right\|_{L^5_x}.
\label{eq:kernel-bound}
\end{equation}
\end{lemma}

\begin{proof}
By definition, the kernel $\mathcal K_{\alpha,\beta,f}(x,y)$ is given by the formula
\begin{align}\label{kernel expression}
\mathcal K_{\alpha,\beta,f}(x,y)
=
\int_{\R^3}
K_\alpha(x-z)f(z)K_\beta(z-y)\,dz.
\end{align}
Next, set
\[
c=\frac{\alpha\beta}{\alpha+\beta},
\qquad
z_*=\frac{\beta x+\alpha y}{\alpha+\beta}.
\]
By direct computation, one easily verifies the identity
\begin{equation}
\frac{|x-z|^2}{4\alpha}
+
\frac{|z-y|^2}{4\beta}
=
\frac{|x-y|^2}{4(\alpha+\beta)}
+
\frac{|z-z_*|^2}{4c}.
\label{eq:complete-square}
\end{equation}
Substituting \eqref{eq:complete-square} into the kernel integral \eqref{kernel expression}, combined with the definition of $K_t(x)$ given in \eqref{def of dispersive kernel}
and the identity $c(\alpha+\beta)=\alpha\beta$, we obtain
\begin{align*}
\mathcal K_{\alpha,\beta,f}(x,y)
&=
K_{\alpha+\beta}(x-y)
\frac{1}{(4\pi i c)^{3/2}}
\int_{\R^3}
e^{i|z-z_*|^2/(4c)}f(z)\,dz
\\
&=
K_{\alpha+\beta}(x-y)
[S(c)f](z_*),
\end{align*}
which proves \eqref{eq:kernel-identity}. Next, for fixed $y$, we may view $\mathcal K_{\alpha,\beta,f}(\cdot,y)$ as a function of $x$. Combining change of variables (between $x$ and $z_*$) and the pointwise dispersive decay
$|K_{\alpha+\beta}(x-y)|
\lesssim |\alpha+\beta|^{-3/2}$, \eqref{eq:kernel-identity} gives
\begin{align*}
\|\mathcal K_{\alpha,\beta,f}(\cdot,y)\|_{L_x^5}
&\lesssim
|\alpha+\beta|^{-3/2}
\left|\frac{\alpha+\beta}{\beta}\right|^{3/5}
\|S(c)f\|_{L_x^5}
\\
&=
|\alpha+\beta|^{-9/10}|\beta|^{-3/5}
\left\|
S\left(\frac{\alpha\beta}{\alpha+\beta}\right)f
\right\|_{L_x^5}.
\end{align*}
The right-hand side is independent of $y$. Thus, for
$g\in L^1(\R^3)$, Minkowski's inequality yields
\[
\|S(\alpha)f S(\beta)g\|_{L^5}
\le
\sup_y
\|\mathcal K_{\alpha,\beta,f}(\cdot,y)\|_{L^5}
\|g\|_{L^1},
\]
which proves \eqref{eq:kernel-bound}.
\end{proof}

Utilizing \eqref{eq:kernel-bound}, the next lemma gives the desired space-time Strichartz estimate to the coefficient. 

\begin{lemma}
\label{lem:double}
Let $J=[\tau,T]$ with $\tau<T$. For $f,g:\R\times\R^3\to\C$, define
\[
Y_+= D_J(fD_\tau g),
\qquad
Y_-= D_J(f\overline{D_\tau g}).
\]
Then
\begin{equation*}
\|Y_\pm\|_{L^5_{t,x}(\R)}
+
\|Y_\pm\|_{L_t^\infty \dot{B}^{-1}_{\infty,\infty}(\R)}
\lesssim
\|g\|_{L_t^\infty L_x^1(J)}\||D|^{3/10}f\|_{L_t^{5/2}L_x^2(J)}
.
\end{equation*}
\end{lemma}

\begin{proof}
It suffices to consider the estimates for $Y_+$, the term $Y_-$ can be similarly dealt with. We may rewrite $Y_+$ by switching the order of the double integrals:
\[
Y_+(t)
=
\int_{\substack{r\in J\\ r<t}}
\int_\tau^r
S(t-r)f(r)S(r-s)g(s)\,ds\,dr.
\]
Set $\alpha=t-r$ and $\beta=r-s$. Lemma~\ref{lem:kernel} then gives
\begin{align}
\begin{aligned}
&\left\|
\int_\tau^r
S(t-r)f(r)S(r-s)g(s)\,ds
\right\|_{L_x^5}
\\
&\qquad\lesssim
\|g\|_{L_t^\infty L_x^1(J)}\int_0^{r-\tau}
(\alpha+\beta)^{-9/10}\beta^{-3/5}
\left\|
S\left(\frac{\alpha\beta}{\alpha+\beta}\right)f(r)
\right\|_{L_x^5}\,d\beta.
\end{aligned}
\label{eq:inner-plus}
\end{align}
Since the integrand is nonnegative, we may enlarge the interval of integration to $(0,\infty)$. By a change of variable in $\beta$, the integral on the r.h.s. of
\eqref{eq:inner-plus} becomes
\begin{equation}
\alpha^{-2/5}
\int_0^\alpha
c^{-3/5}(\alpha-c)^{-1/2}
\|S(c)f(r)\|_{L_x^5}\,dc.
\label{eq:c-integral-plus}
\end{equation}
Applying H\"older with exponents $10/7$ and $10/3$ yields
\[
\eqref{eq:c-integral-plus}
\lesssim
\alpha^{-2/5}
\left(
\int_0^\alpha
c^{-6/7}(\alpha-c)^{-5/7}\,dc
\right)^{7/10}
\|S(c)f(r)\|_{L_c^{10/3}L_x^5(\R)}.
\]
By direct computation one easily verifies that
\[\left(
\int_0^\alpha
c^{-6/7}(\alpha-c)^{-5/7}\,dc
\right)^{7/10}\lesssim \alpha^{-\frac25}.\]
Combining Strichartz and Sobolev it follows
\[
\eqref{eq:c-integral-plus}
\lesssim
\alpha^{-4/5}
\|S(c)f(r)\|_{L_c^{10/3}L_x^5(\R)}\lesssim 
\alpha^{-4/5}
\|S(c)|D|^{\frac{3}{10}}f(r)\|_{L_c^{10/3}L_x^{\frac{10}{3}}(\R)}\lesssim \alpha^{-4/5}\||D|^{\frac{3}{10}}f(r)\|_{L_x^2}.
\]
Summing up, we obtain
\begin{equation*}
\left\|
\int_\tau^r
S(t-r)f(r)S(r-s)g(s)\,ds
\right\|_{L_x^5}
\lesssim
\|g\|_{L_t^\infty L_x^1(J)}(t-r)^{-4/5}\||D|^{3/10}f(r)\|_2
\end{equation*}
and consequently
\[
\|Y_+(t)\|_{L_x^5}
\lesssim
\|g\|_{L_t^\infty L_x^1(J)}\int_{r<t}
(t-r)^{-4/5}\mathbf 1_J(r)
\||D|^{3/10}f(r)\|_2\,dr.
\]
The desired estimate for $\|Y_+\|_{L^5_{t,x}(\R)}$ follows now from applying the one-dimensional Hardy-Littlewood-Sobolev inequality.

It remains to show the uniform in time Besov estimate for $\|Y_+\|_{L_t^\infty \dot{B}^{-1}_{\infty,\infty}(\R)}$. Fix
$t_*\in\R$ and define 
\[
f^{t_*}(r):=(\chi_{r\leq t_*}f)(r),\qquad (Y_+)^{t_*}(r)
=D_J(f^{t_*}D_\tau g)(r).
\]
Then arguing as in \eqref{3.5} we obtain
\begin{equation*}
\begin{aligned}
\|Y_+(t_*)\|_{\dot{B}^{-1}_{\infty,\infty}}
&\lesssim
\|S(\cdot )Y_+(t_*)\|_{L^5_{t,x}(0,\infty)}=\|(Y_+)^{t_*}(t_*+\cdot)\|_{L^5_{t,x}(0,\infty)}
\\
&\lesssim \|(Y_+)^{t_*}\|_{L_{t,x}^5(\R)}\lesssim
\|g\|_{L_t^\infty L_x^1(J)}\||D|^{3/10}f\|_{L_t^{5/2}L_x^2(J)}.
\end{aligned}
\end{equation*}
Taking the supremum over $t_*\in\R$ yields the desired claim.
\end{proof}

\section{Proof of Theorem \ref{main thm}}\label{sec:main-proof}
Having all the preliminaries, we are now ready to give the proof of our main result.

\begin{proof}[Proof of Theorem \ref{main thm}]
By symmetry, we may simply assume that $J$ is of the form $[\tau,T]$.

Assume $u(\tau)=v(\tau)$ for some $\tau\in J$. Recall from the classical local theory of the cubic NLS (see e.g. \cite{Cazenave1990}) that there exists some $0<\delta_0\ll 1$ such that \eqref{cnls} admits a strong solution $U$ satisfying

\[
U\in C_t\dot{H}^{\frac12}_x(J_{\delta_0})
   \cap L_{t,x}^5(J_{\delta_0}),
\quad
|D|^{1/2}U
\in L_t^{5/2}L_x^{30/7}(J_{\delta_0}),\quad U(\tau)=u(\tau),
\]
where $J_{\delta_0}=[\tau,\tau+\delta_0]\subset J$. 

Our first aim is to prove that $u(t)=U(t)$ for all $t\in J_\delta=[\tau,\tau+\delta]$ for some $\delta\in(0,\delta_0)$. Set $w=u-U$. Then $w$ satisfies the difference NLS
\[(i\pt_t+\Delta)w =Aw+B\overline w+\mu|w|^2 w,\]
where 
\[
A:=\mu\bigl(2|U|^2+2\operatorname{Re}(U\overline w)\bigr),\qquad
B:=\mu(U^2+Uw)
\]
and $w(\tau)=0$. This yields the Duhamel formulation
\begin{align}\label{5.1}
w(t)=-i D_\tau(Aw+B\overline w)-iD_\tau(\mu|w|^2 w)
\end{align} 
for $w$ at $t\in J_{\delta_0}$. By Lemma \ref{lem 3.1} and Lemma \ref{lem:refined-sobolev}, we have
\begin{align*}
\begin{aligned}
\|D_\tau(\mu|w|^2 w)\|_{L_t^\infty \dot{B}^{-1}_{\infty,\infty}(J_\delta)}&\lesssim \|w^3\|_{L_t^\infty L_x^1(J_\delta)}
=\|w\|^3_{L_t^\infty L_x^3(J_\delta)}\\
&\lesssim \|w\|_{L_t^\infty\dot H_x^{1/2}(J_\delta)}^{2}
    \|w\|_{L_t^\infty{\dot{B}^{-1}_{\infty,\infty}}(J_\delta)}.
\end{aligned} 
\end{align*}
By $w(\tau)=0$ and the continuity of $w$, for any $\varepsilon$ we can choose $\delta$ sufficiently small such that 
\[\|w\|_{L_t^\infty\dot H_x^{1/2}(J_\delta)}^{2}\|w\|_{L_t^\infty{\dot{B}^{-1}_{\infty,\infty}}(J_\delta)}\leq \varepsilon
\|w\|_{L_t^\infty{\dot{B}^{-1}_{\infty,\infty}}(J_\delta)}.\]
Absorbing this term to the l.h.s. of \eqref{5.1} yields
\begin{align}\label{5.1+}
\|w\|_{L_t^\infty{\dot{B}^{-1}_{\infty,\infty}}(J_\delta)}\lesssim
\|D_\tau(Aw+B\overline w)\|_{L_t^\infty{\dot{B}^{-1}_{\infty,\infty}}}.
\end{align}
The term $\|D_\tau(Aw+B\overline w)\|_{L_t^\infty{\dot{B}^{-1}_{\infty,\infty}}}$ will be estimated via suitable Born expansion argument. 

In the following we shall fix the time interval $J_\delta$ and ignore the dependence of the function spaces on this interval. First define
\[Kf:=-iD_\tau(Af+B\overline f),\quad z_0:=K(-iD_\tau(\mu|w|^2 w)),\quad r:=Kw\]
Applying the operator $K$ to both sides of \eqref{5.1} yields
\begin{equation}\label{rkz}
r=Kr+z_0.
\end{equation}
Using H\"older it follows
\[
\begin{aligned}
\|U^2\|_{L_t^5L_x^{15/8}}
\le
\|U\|_{L_t^5L_x^5}
\|U\|_{L_t^\infty L_x^3},\\
\|U w\|_{L_t^5L_x^{15/8}}
\le
\|U\|_{L_t^5L_x^5}
\|w\|_{L_t^\infty L_x^3},
\end{aligned}
\]
hence yielding 
\(
A,B\in L_t^5L_x^{15/8}.
\)
Another application of H\"older, combing the exotic Strichartz \eqref{2.4}, gives
\[
\begin{aligned}
\|r\|_{L_{t,x}^5}&= \|D_\tau(Aw+B\overline w)\|_{L_{t,x}^5}
\lesssim \|Aw+B\overline w\|_{L_t^5 L_x^{\frac{15}{13}}} \\
&\lesssim 
(\|A\|_{L_t^5 L_x^{15/8}}+\|B\|_{L_t^5 L_x^{15/8}})
\|w\|_{L_t^\infty L_x^3}
\end{aligned}
\]
and we conclude that $r\in L_{t,x}^5$. Another similar application of H\"older and exotic Strichartz yields again
\[
\begin{aligned}
\|Kr\|_{L_{t,x}^5}&= \|D_\tau(Ar+B\overline r)\|_{L_{t,x}^5}
\lesssim \|Ar+B\overline r\|_{L_t^{\frac{5}{2}} L_x^{\frac{15}{11}}} \\
&\lesssim 
(\|A\|_{L_t^5 L_x^{15/8}}+\|B\|_{L_t^5 L_x^{15/8}})
\|r\|_{L_{t,x}^5}.
\end{aligned}
\]
By further shrinking $\delta$ if necessary, we may assume that 
\[\|A\|_{L_t^5 L_x^{15/8}}+\|B\|_{L_t^5 L_x^{15/8}}\leq \varepsilon\]
so that the term $Kr$ in \eqref{rkz} can be absorbed to the l.h.s. We thus conclude that
\begin{align}\label{5.4}
\|r\|_{L_{t,x}^5}\lesssim \|z_0\|_{L_{t,x}^5}.
\end{align}
We obtain from Lemma \ref{lem:double} and Sobolev embedding that
\begin{equation}\label{5.5}
\begin{aligned}
\|z_0\|_{L_{t,x}^5}+\|z_0\|_{L_{t}^\infty \dot{B}^{-1}_{\infty,\infty}}&\lesssim \|w^3\|_{L_t^\infty L_x^1}(\||D|^{\frac{3}{10}}A\|_{L_t^{5/2}L_x^2}+
\||D|^{\frac{3}{10}}B\|_{L_t^{5/2}L_x^2})\\
&\lesssim \|w\|^3_{L_t^\infty L_x^3}(\||D|^{\frac{3}{10}}A\|_{L_t^{5/2}L_x^2}+
\||D|^{\frac{3}{10}}B\|_{L_t^{5/2}L_x^2}).
\end{aligned}
\end{equation}
By definition of $A$ and $B$,
\[
\begin{aligned}
&\bigl\||D|^{3/10}A\bigr\|_{L_t^{5/2}L_x^2}+\bigl\||D|^{3/10}B\bigr\|_{L_t^{5/2}L_x^2}\\
&\qquad\lesssim
\bigl\||D|^{3/10}(U^2)\bigr\|_{L_t^{5/2}L_x^2}
+\bigl\||D|^{3/10}(|U|^2)\bigr\|_{L_t^{5/2}L_x^2}\\
&\qquad\qquad+\bigl\||D|^{3/10}(U w)\bigr\|_{L_t^{5/2}L_x^2}
+\bigl\||D|^{3/10}(U\overline w)\bigr\|_{L_t^{5/2}L_x^2}.
\end{aligned}
\]
We estimate the above terms in order. By fractional Leibniz rule, for $v=U,\overline U,w,\overline w$,
\[
\begin{aligned}
\bigl\||D|^{3/10}(Uv)\bigr\|_{L_x^2}
&\lesssim
\bigl\||D|^{3/10}U\bigr\|_{L_x^6}\|v\|_{L_x^3}
+
\|U\|_{L_x^{15}}
\bigl\||D|^{3/10}v\bigr\|_{L_x^{30/13}}.
\end{aligned}
\]
Using the Sobolev embeddings
\[
\begin{aligned}
\bigl\||D|^{3/10}U\bigr\|_{L_x^6}
+
\|U\|_{L_x^{15}}
&\lesssim
\bigl\||D|^{1/2}U\bigr\|_{L_x^{30/7}},\\
\|v\|_{L_x^3}
+
\bigl\||D|^{3/10}v\bigr\|_{L_x^{30/13}}
&\lesssim
\|v\|_{\dot H_x^{1/2}},
\end{aligned}
\]
we obtain by combining H\"older in time that
\[
\bigl\||D|^{3/10}(Uv)\bigr\|_{L_t^{5/2}L_x^2}
\lesssim
\bigl\||D|^{1/2}U\bigr\|_{L_t^{5/2}L_x^{30/7}}
\|v\|_{L_t^\infty\dot H_x^{1/2}}.
\]
Applying this with $v=U, \overline U, w,\overline w$ and inserting to \eqref{5.5} yields
\begin{align}\label{5.6+}
\begin{aligned}
\|z_0\|_{L_{t,x}^5}+\|z_0\|_{L_{t}^\infty \dot{B}^{-1}_{\infty,\infty}}
\lesssim \|w\|^3_{L_t^\infty L_x^3}
\bigl\||D|^{1/2}U\bigr\|_{L_t^{5/2}L_x^{30/7}}
\left(
\|U\|_{L_t^\infty\dot H_x^{1/2}}
+
\|w\|_{L_t^\infty\dot H_x^{1/2}}
\right).
\end{aligned}
\end{align}
Now \eqref{5.1+}, Corollary \ref{cor 3.3}, H\"older, \eqref{5.4}, \eqref{5.6+} and Lemma \ref{lem:refined-sobolev} imply
\begin{equation*}
\begin{aligned}
&\|w\|_{L_t^\infty{\dot{B}^{-1}_{\infty,\infty}}}\lesssim 
\|r\|_{L_t^\infty{\dot{B}^{-1}_{\infty,\infty}}}\lesssim 
\|Ar+B\overline r\|_{L_t^{5/2}L_x^{15/11}}+\|z_0\|_{L_{t}^\infty \dot{B}^{-1}_{\infty,\infty}}\\
&\qquad\lesssim (\|A\|_{L_t^5 L_x^{15/8}}+\|B\|_{L_t^5 L_x^{15/8}})\|r\|_{L_{t,x}^5}
+\|z_0\|_{L_{t}^\infty \dot{B}^{-1}_{\infty,\infty}}\\
&\qquad\lesssim (1+\|A\|_{L_t^5 L_x^{15/8}}+\|B\|_{L_t^5 L_x^{15/8}})(\|z_0\|_{L_{t,x}^5}+\|z_0\|_{L_{t}^\infty \dot{B}^{-1}_{\infty,\infty}})\\
&\qquad \lesssim (1+\|A\|_{L_t^5 L_x^{15/8}}+\|B\|_{L_t^5 L_x^{15/8}})
\bigl\||D|^{1/2}U\bigr\|_{L_t^{5/2}L_x^{30/7}}
\\
&\qquad\qquad\times \left(
\|U\|_{L_t^\infty\dot H_x^{1/2}}
+
\|w\|_{L_t^\infty\dot H_x^{1/2}}
\right)\|w\|^3_{L_t^\infty L_x^3}\\
&\qquad \lesssim (1+\|A\|_{L_t^5 L_x^{15/8}}+\|B\|_{L_t^5 L_x^{15/8}})
\bigl\||D|^{1/2}U\bigr\|_{L_t^{5/2}L_x^{30/7}}
\\
&\qquad\qquad \times \left(
\|U\|_{L_t^\infty\dot H_x^{1/2}}
+
\|w\|_{L_t^\infty\dot H_x^{1/2}}
\right)\|w\|_{L_t^\infty\dot H_x^{1/2}}^{2}
    \|w\|_{L_t^\infty{\dot{B}^{-1}_{\infty,\infty}}}\lesssim \vare \|w\|_{L_t^\infty{\dot{B}^{-1}_{\infty,\infty}}}
\end{aligned}
\end{equation*}
by shrinking $\delta$ if necessary. We then conclude that $\|w\|_{L_t^\infty{\dot{B}^{-1}_{\infty,\infty}}}\leq \frac12 \|w\|_{L_t^\infty{\dot{B}^{-1}_{\infty,\infty}}}$ for $\vare$ sufficiently small, from which we conclude that $w=0$ on $J_\delta$, and the proof of the claim $u=U$ on $J_\delta$ follows. Consequently, we conclude that $u=U=v$ on $J_\delta$

Consider now the set 
\[\mathcal{E}:=\{t\in J:\,u(t)=v(t)\}.\]
By the continuity of $u$ and $v$ and the above claim, we know that $\mathcal{E}$ must be open and closed simultaneously. Since $J$ is connected, this implies $\mathcal{E}=J$, as desired.
\end{proof}

\section{AI disclosure}
The proof of the main result in this paper was developed with substantive assistance from GPT-6 Astra (OpenAI). After studying
the work of Han and Fang \cite{HanFang2013}, we wished to investigate the endpoint unconditional uniqueness problem for the
three-dimensional cubic NLS in $\dot H^{1/2}(\mathbb R^3)$. We asked the model to explore whether their approach could be extended to this endpoint by identifying a suitable Besov space in which to close the uniqueness argument. The initial draft of the proof was generated by GPT-6 Astra in response to this request.

The presentation of that draft was difficult to follow and required substantial clarification. We subsequently worked through the
argument in detail and reorganized and rewrote the proof according to our own understanding. One important change concerned the
treatment of the Besov norm and the associated estimates, which were originally formulated using the heat semigroup. We recast
these definitions and proofs in terms of Littlewood-Paley projections, with the aim of presenting the argument in a framework
more familiar to the dispersive PDE community. The present manuscript is the outcome of this process of studying, checking,
and revising the initial argument.

GPT-6 Astra was also used in subsequent discussions of the proof and during the preparation and revision of the manuscript,
including for grammar checking and language editing. All mathematical arguments in the final manuscript have been checked
by the authors. We have reviewed and approved the final text and take full responsibility for the correctness and integrity of
the work, including all content developed with AI assistance.

\subsubsection*{Acknowledgements}
Y.L. was supported by the NSF grant of Guangdong (No. 2024A1515010497), the QB-Program of Guangdong (No. 2024QN11X141) and the NSF grant of China (No. 12301301).

\subsubsection*{Data availability}
Data sharing is not applicable to this article as no datasets were generated or analysed during the current study.

\subsubsection*{Conflict of interest}
The authors declare that they have no conflict of interest.


\end{document}